\documentclass[11pt,reqno]{amsart}
\usepackage[T1]{fontenc}
\usepackage{lmodern,microtype,amsmath,amssymb,mathtools}
\usepackage[a4paper,margin=29mm]{geometry}
\usepackage[colorlinks=true,linkcolor=blue!45!black,citecolor=blue!45!black,urlcolor=blue!45!black]{hyperref}
\usepackage{xcolor}
\newtheorem{theorem}{Theorem}[section]
\newtheorem{proposition}[theorem]{Proposition}
\newtheorem{lemma}[theorem]{Lemma}
\newtheorem{corollary}[theorem]{Corollary}
\theoremstyle{definition}

\theoremstyle{remark}
\newtheorem{remark}[theorem]{Remark}
\newcommand{\T}{\mathbb T}

\newcommand{\Z}{\mathbb Z}
\newcommand{\C}{\mathbb C}
\newcommand{\B}{\mathcal B}
\newcommand{\D}{\mathcal D}
\newcommand{\G}{\mathcal G}
\newcommand{\I}{\mathcal I}
\newcommand{\DP}{\operatorname{DP}}
\newcommand{\Id}{\operatorname{Id}}
\newcommand{\supp}{\operatorname{supp}}
\newcommand{\dist}{\operatorname{dist}}
\newcommand{\clspan}{\overline{\operatorname{span}}}

\newcommand{\restr}[2]{#1\!\upharpoonright_{#2}}
\title[Large ideals in $\B(L^1(0,1))$]{Large ideals in $\B(L^1(0,1))$}
\author{Amir Nasseri}
\date{September 16, 2026}
\begin{document}
\begin{abstract}
We show that the norm-closed linear span of products of two Dunford--Pettis operators on $L^1(0,1)$ lies strictly between the representable and the Dunford--Pettis operators. This gives an additional proper large closed ideal in $\B(L^1(0,1))$, answering the existence question of Johnson, Pisier and Schechtman. The quotient of the Dunford--Pettis ideal by this square ideal contains a contractively complemented isometric copy of $L^1(0,1)$. We also prove that the Dunford--Pettis ideal has no right approximate identity, answering a question of Johnson and Schechtman. More precisely, a positive norm-one convolution operator $R$ satisfies $\|R-RT\|\geq1$ and $\|R-TR\|\geq1$ for every Dunford--Pettis operator $T$. The proofs combine translation averaging, Rajchman measures supported on a strongly independent set, and two atomless disintegrations of positive operators. The approximate-identity obstruction extends to every finite measure space with a nonzero atomless part. Finally, we prove that quantitative multiplication estimates for an ideal pass to all its closed power ideals.
\end{abstract}
\maketitle

\section{Introduction and main results}

Johnson and Schechtman \cite{JS} studied one-sided approximate identities in closed ideals of $\B(L^p[0,1])$. They proved that the compact operators form the only nonzero proper closed ideal with a left approximate identity. For $p=1$, they also proved that the representable operators have a contractive right approximate identity. At the beginning of Section~3 of their paper they asked whether the Dunford--Pettis ideal has a right approximate identity and proposed a negative answer. We prove that negative answer, with an exact norm obstruction for one fixed operator.

Throughout the paper the scalar field is $\C$, $\T=\mathbb R/\mathbb Z$, and $m$ is normalized Haar measure. We work first on $X=L^1(\T,m)$, which is isometrically isomorphic to $L^1[0,1]$. An operator is \emph{Dunford--Pettis} if it maps weakly null sequences to norm-null sequences. Write
\[
 \D=\DP(X),\qquad
 \G=\G_{\ell^1}(X),\qquad
 \I_2=\clspan\{ST:S,T\in\D\}.
\]
Here $\G$ is the closed ideal of representable operators; on this space it is also the ideal of operators factoring through $\ell^1$. In particular, $\G\subseteq\D$. All operator-ideal closures and quotient norms use the operator norm; measure spaces carry the total variation norm. A right approximate identity in a Banach algebra $J$ is a net $(e_\alpha)\subseteq J$ for which $\|ae_\alpha-a\|\to0$ for every $a\in J$. No boundedness of the net is part of this definition.

\begin{theorem}\label{thm:ai}
There is a positive operator $R\in\D$ with $\|R\|=1$ such that
\begin{equation}\label{eq:obstruction}
 \|R-RT\|\geq1,\qquad \|R-TR\|\geq1\qquad(T\in\D).
\end{equation}
Consequently, $\DP(L^1[0,1])$ has neither a right nor a left approximate identity, with no boundedness assumption.
\end{theorem}

The assertion concerning left approximate identities was already known from \cite[Theorem~2.1]{JS}. The new application to their question concerns the right side. The same convolution operator witnesses both estimates in \eqref{eq:obstruction}.

For a finite complex Borel measure $\mu$ on $\T$, let
\[
 C_\mu f(x)=\int_\T f(x-y)\,d\mu(y),\qquad
 \widehat\mu(n)=\int_\T e^{-2\pi int}\,d\mu(t).
\]
The Rajchman algebra $M_0(\T)$ consists of measures for which $\widehat\mu(n)\to0$ as $|n|\to\infty$. For a compact set $P$, put
\[
 M_0(P)=\{\mu\in M_0(\T):\supp\mu\subseteq P\}.
\]

\begin{theorem}\label{thm:square}
There is a compact perfect strongly independent set $P\subseteq\T$ supporting a Rajchman probability measure such that
\begin{equation}\label{eq:exact-distance}
 \dist(C_\mu,\I_2)=\|\mu\|\qquad(\mu\in M_0(P)).
\end{equation}
Moreover,
\begin{equation}\label{eq:strict}
 \G\subsetneq\I_2\subsetneq\D.
\end{equation}
The map $\mu\mapsto C_\mu+\I_2$ embeds $M_0(P)$ isometrically as a contractively complemented subspace of $\D/\I_2$. This quotient also contains a contractively complemented isometric copy of $L^1[0,1]$.
\end{theorem}

Thus $\I_2$ is a canonical intermediate closed operator ideal. It is large in the usual sense of not being contained in the strictly singular operators. This also answers a second question. Johnson, Pisier and Schechtman \cite[Introduction]{JPS} constructed a continuum of small closed ideals and asked whether $\B(L^1[0,1])$ has more than the three classical proper large closed ideals: the representable ideal, the Dunford--Pettis ideal, and the unique maximal proper ideal. The next consequence of Theorem~\ref{thm:square} gives an affirmative answer.

\begin{corollary}[An additional proper large closed ideal]\label{cor:large}
Let $\mathcal M(X)$ denote the unique maximal proper ideal of $\B(X)$. Then
\[
 \G\subsetneq\I_2\subsetneq\D\subsetneq\mathcal M(X)\subsetneq\B(X).
\]
In particular, $\B(L^1[0,1])$ has at least four distinct proper large closed two-sided ideals.
\end{corollary}
\begin{proof}
The classical strict inclusion $\D\subsetneq\mathcal M(X)$ is recalled in \cite[Introduction]{JPS}; the two strict inclusions involving $\I_2$ are Theorem~\ref{thm:square}. To see directly that all four ideals are large, choose bounded operators $U:\ell^1\to X$ and $V:X\to\ell^1$ with $VU=\Id_{\ell^1}$. Then $UV\in\G$ is a projection onto an infinite-dimensional copy of $\ell^1$. Its restriction to its range is the identity, so it is not strictly singular. Each of the four ideals contains this projection.
\end{proof}

Acuaviva's recent preprint \cite{Acuaviva} constructs complemented subspaces of $L^1[0,1]$ using biased product measures and resolvent averages of convolution operators. In personal communication, Acuaviva has informed the author that related constructions yield $2^{\mathfrak c}$ distinct norm-closed two-sided ideals in $\B(L^1(0,1))$, where $\mathfrak c=2^{\aleph_0}$. This conclusion is also announced in the introduction of \cite{Acuaviva}, with details deferred to a subsequent version. The present paper follows a different approach. Its central ingredients are translation averaging and atomless disintegrations, combined with a strongly independent set supporting a Rajchman measure. These identify the closed square of the Dunford--Pettis ideal as a canonical proper large ideal and give an exact distance formula for a family of convolution operators.

Our construction establishes the existence of an additional large ideal; it does not assert that there are infinitely many large closed ideals. If
\[
 A=\B(X)/\G,\qquad D=\D/\G,\qquad
 D^{[n]}=\clspan\{d_1\cdots d_n:d_i\in D\},
\]
then \eqref{eq:strict} is equivalently $0\ne D^{[2]}\subsetneq D$, and $\D/\I_2$ is canonically isometric to $D/D^{[2]}$.

The harmonic-analysis input has a classical origin. Varopoulos constructed strongly independent sets supporting Rajchman measures \cite{VarSets,VarDecomp}. Ghandehari used these sets to prove that $M_0(G)$ has no approximate identity for a nondiscrete locally compact abelian group \cite[Theorem~3.3]{Gh12}. We use the same obstruction. The step that gives Theorem~\ref{thm:ai} is a contractive convolution-bimodule projection from the Dunford--Pettis operators onto the Rajchman convolution operators. Theorem~\ref{thm:square} requires a further fact: averaging an arbitrary product of two Dunford--Pettis operators produces a measure whose restriction to $P$ vanishes. The two factors need not commute with translations.

We give the shorter proof of Theorem~\ref{thm:ai} first. The square-ideal theorem is proved in Sections~\ref{sec:kernels}--\ref{sec:square}. Section~\ref{sec:transfer} extends the approximate-identity obstruction to finite measure spaces with an atomless part. Section~\ref{sec:norming} treats quantitative power-ideal estimates as a separate abstract statement. In particular, neither main theorem depends on a quantitative detection theorem or on an unpublished companion manuscript.

\section{Rajchman measures and translation averaging}\label{sec:averaging}

We record the elementary properties of Rajchman measures used below, including the convolution criterion for Dunford--Pettis operators.

\begin{lemma}\label{lem:rajchman}
The space $M_0(\T)$ is a closed convolution ideal in $M(\T)$. If $\mu\in M_0(\T)$ and $\nu\ll|\mu|$, then $\nu\in M_0(\T)$. Every Rajchman measure is atomless.
\end{lemma}
\begin{proof}
Closedness and the convolution-ideal property follow from the Fourier transform and $|\widehat\mu(n)|\leq\|\mu\|$. If $\mu\in M_0$ and $h\in L^1(|\mu|)$, approximate $h$ in $L^1(|\mu|)$ by trigonometric polynomials. Each polynomial multiple of $\mu$ is Rajchman, since its Fourier coefficients are finite linear combinations of translates of $\widehat\mu$. Thus $h\mu\in M_0$. Polar decomposition now implies the assertion for all $\nu\ll|\mu|$. In particular, restricting a Rajchman measure to a singleton again gives a Rajchman measure. A nonzero point mass is not Rajchman, so all such restrictions vanish.
\end{proof}

\begin{lemma}\label{lem:convolution}
For $\mu\in M(\T)$, $\|C_\mu\|_{\B(L^1)}=\|\mu\|$. Moreover,
\[
 C_\mu\in\D\quad\Longleftrightarrow\quad\mu\in M_0(\T).
\]
\end{lemma}
\begin{proof}
The upper norm estimate is Young's inequality. Let $(h_k)$ be a nonnegative $L^1$ approximate identity with $\|h_k\|_1=1$. The measures $(C_\mu h_k)m$ converge weak-star to $\mu$. Lower semicontinuity of the variation norm gives $\|\mu\|\leq\liminf_k\|C_\mu h_k\|_1$, proving equality.

For necessity, the characters $e_n(x)=e^{2\pi inx}$ are weakly null in $L^1$ as $|n|\to\infty$, by the Riemann--Lebesgue lemma applied to $L^\infty\subseteq L^1$. Since $C_\mu e_n=\widehat\mu(n)e_n$, the Dunford--Pettis property implies $\widehat\mu(n)\to0$.

Conversely, if $\mu\in M_0$, then $C_\mu$ is compact on $L^2$: in the Fourier basis it is diagonal with diagonal in $c_0(\Z)$. A bounded subset of $L^\infty$ is bounded in $L^2$, so its image is relatively compact in $L^1$. Any relatively weakly compact subset $K\subseteq L^1$ is uniformly integrable. Truncation therefore approximates $K$, uniformly in $L^1$, by bounded subsets of $L^\infty$. Boundedness of $C_\mu$ shows that $C_\mu K$ is totally bounded. Hence $C_\mu$ maps relatively weakly compact sets to relatively norm-compact sets and is Dunford--Pettis.
\end{proof}

For $s\in\T$, define $U_sf(x)=f(x-s)$. The translation group is strongly continuous on $L^1$. For $T\in\B(X)$ set
\begin{equation}\label{eq:average}
 E(T)f=\int_\T U_{-s}TU_sf\,dm(s).
\end{equation}
This is a Bochner integral in $X$ for each fixed $f$; it is not asserted to be a Bochner integral in $\B(X)$.

\begin{proposition}\label{prop:average}
The map $E$ is a contractive projection onto the translation-invariant operators. There is a linear contraction $\sigma:\B(X)\to M(\T)$ such that
\[
 E(T)=C_{\sigma(T)},\qquad \sigma(C_\mu)=\mu.
\]
Furthermore,
\begin{align}
 E(\D)&\subseteq\D,& \sigma(\D)&\subseteq M_0(\T),\label{eq:preserves}\\
 E(C_\mu T)&=C_\mu E(T),&E(TC_\mu)&=E(T)C_\mu \label{eq:module}
\end{align}
for $\mu\in M(\T)$ and $T\in\B(X)$.
\end{proposition}
\begin{proof}
For fixed $f$, the integrand in \eqref{eq:average} is norm-continuous and bounded in norm by $\|T\|\|f\|_1$. Thus $E$ is well defined and contractive. Invariance of Haar measure gives $U_{-t}E(T)U_t=E(T)$ for every $t$. Operators commuting with all translations are fixed by $E$. Wendel's multiplier theorem \cite{Wendel} identifies these operators isometrically with convolution by measures, giving $\sigma$.

If $T\in\D$ and $f_n\rightharpoonup0$ in $L^1$, then for each fixed $s$, $\|TU_sf_n\|_1\to0$. The sequence $(f_n)$ is norm-bounded, and dominated convergence gives
\[
 \|E(T)f_n\|_1\leq\int_\T\|TU_sf_n\|_1\,dm(s)\longrightarrow0.
\]
Lemma~\ref{lem:convolution} now yields \eqref{eq:preserves}. Finally, $C_\mu$ commutes with every $U_s$; moving it through the vector integral proves \eqref{eq:module}.
\end{proof}

\section{A strongly independent set and the approximate-identity obstruction}\label{sec:polar}

A set $P$ in an abelian group is \emph{strongly independent} if every relation $\sum_{j=1}^r n_jp_j=0$ among distinct points of $P$, with integer coefficients, forces every $n_j$ to be divisible by the torsion parameter $k(P)$. This parameter is the least positive integer $k$ for which $kP=\{0\}$, if one exists, and is infinite otherwise. Divisibility by infinity means that the coefficient is zero.

\begin{theorem}[Varopoulos; see Ghandehari]\label{thm:varopoulos}
There is a compact perfect strongly independent set $P\subseteq\T$ supporting a positive Rajchman probability measure. For distinct $x,y\in\T$,
\begin{equation}\label{eq:intersection}
 |(x+P)\cap(y+P)|\leq2.
\end{equation}
\end{theorem}
We use the formulation in \cite[Lemmas~3.1--3.2]{Gh12}; see also \cite[Lemma~2.7 and Lemma~4.1]{Gh19}. Since a finite torsion subgroup of $\T$ is finite and $P$ is perfect, $k(P)=\infty$. Thus the points of $P$ are linearly independent over $\Z$.

\begin{lemma}[The polar-set calculation]\label{lem:polar}
If $\alpha,\beta$ are finite positive atomless measures, then
\[
 (\alpha*\beta)(P)=0.
\]
The same holds with $\beta$ replaced by its reflected measure $\check\beta(B)=\beta(-B)$. Consequently $m(P)=0$. For arbitrary complex atomless measures, the convolution has zero restriction to $P$.
\end{lemma}
\begin{proof}
By \eqref{eq:intersection}, the sets $P-x$, $x\in\T$, are pairwise disjoint modulo $\beta$-null sets. In a finite measure space only countably many members of such a family can have positive measure. Indeed, for each $k$ there are only finitely many with measure at least $1/k$, by finite additivity modulo null sets. Thus
\[
 N_\beta=\{x:\beta(P-x)>0\}
\]
is countable. Tonelli's theorem and atomlessness of $\alpha$ give
\[
 (\alpha*\beta)(P)=\int_\T\beta(P-x)\,d\alpha(x)=0.
\]
Reflection preserves atomlessness. Taking $\alpha=\beta=m$ proves $m(P)=0$. For complex measures use $|\alpha*\beta|\leq|\alpha|*|\beta|$.
\end{proof}

This is the obstruction used in Ghandehari's proof that $M_0(\T)$ has no approximate identity \cite[Theorem~3.3]{Gh12}. Translation averaging transports it to the whole operator ideal.

\begin{proof}[Proof of Theorem~\ref{thm:ai}]
Choose a Rajchman probability $\mu$ supported on $P$, and put $R=C_\mu$. By Lemma~\ref{lem:convolution}, $R$ is positive, Dunford--Pettis, and has norm one. For $T\in\D$, Proposition~\ref{prop:average} gives
\[
 \sigma(RT)=\mu*\sigma(T),\qquad
 \sigma(TR)=\sigma(T)*\mu,
\]
where $\sigma(T)\in M_0(\T)$. Both factors are atomless. Lemma~\ref{lem:polar} implies that the two product symbols vanish on $P$. Since $\sigma$ and restriction to $P$ are contractions,
\[
 \|R-RT\|\geq
 \|\restr{\sigma(R-RT)}{P}\|=\|\mu\|=1.
\]
The other inequality is identical. Neither a right nor a left approximate identity can approximate this fixed $R$. The choice $T=0$ also shows that both infima in \eqref{eq:obstruction} equal one.
\end{proof}

\section{Two atomless disintegrations}\label{sec:kernels}

To separate $C_\mu$ from all of $\I_2$, one must treat products whose two factors are arbitrary Dunford--Pettis operators. Averaging is a convolution-bimodule map, but is not asserted to be multiplicative. We therefore use the operator measures of the two factors.

For a positive $R\in\B(X)$ there is a finite positive Borel measure $\Gamma_R$ on $\T^2$ satisfying
\begin{equation}\label{eq:opmeasure}
 \Gamma_R(A\times B)=\int_A R1_B\,dm,
 \qquad
 \int gRf\,dm=\int_{\T^2}g(x)f(y)\,d\Gamma_R(x,y)
\end{equation}
for bounded Borel $f,g$. Its marginals are
\begin{equation}\label{eq:marginals}
 \Gamma_R(A\times\T)=\int_A R1\,dm,\qquad
 \Gamma_R(\T\times B)=\int_B R^*1\,dm.
\end{equation}
One construction applies the Riesz representation theorem, pointwise off a common null set, to $R^*$ on a countable dense rational vector subspace of $C(\T)$; the resulting positive kernels have masses at most $\|R\|$. Equations~\eqref{eq:opmeasure} then follow first for continuous functions and then for bounded Borel functions by a monotone-class argument. This also gives the usual positive operator-measure representation; compare \cite[Section~1]{Liu}.

Disintegration on compact metric spaces, followed by multiplication by the marginal densities, gives measurable finite positive kernels with
\begin{equation}\label{eq:disint}
 d\Gamma_R(x,y)=dm(x)\,d\beta_x^R(y)
                 =d\alpha_y^R(x)\,dm(y).
\end{equation}
In particular,
\[
 \beta_x^R(\T)=R1(x)\quad\text{a.e.},\qquad
 \alpha_y^R(\T)=R^*1(y)\leq\|R\|\quad\text{a.e.}
\]
Kernels are taken on the completed base measure space. All measure-valued integrals below are understood weak-star, or equivalently as integrals of kernels; norm measurability with values in $M(\T)$ is not needed.

We use two established lattice facts. Every real bounded operator on $L^1$ is regular. The Dunford--Pettis operators on $L^1$ form a solid sublattice, by Bourgain \cite{Bourgain}; Liu \cite[Lemma~2.17]{Liu} records the stronger band property. Thus the positive and negative parts of a real Dunford--Pettis operator are Dunford--Pettis, and $0\leq S\leq R\in\D$ implies $S\in\D$. For a complex operator, apply this to its real and imaginary parts on the real $L^1$ space.

We also use measurable selection of atoms: if a measurable field of finite positive measures has an atomic part on a set of positive base measure, then on a measurable set $F$ of positive measure there are $\varepsilon>0$, a measurable $a\geq\varepsilon$, and a measurable point map $\tau$ with $a(t)\delta_{\tau(t)}$ dominated by that field. This follows from measurable enumeration of the atoms; see \cite[Theorems~1.3--1.4]{Liu}. Completion of the base measure suffices for the measurable selections.

\begin{lemma}\label{lem:atomless}
If $R\geq0$ is Dunford--Pettis, both $\beta_x^R$ and $\alpha_y^R$ in \eqref{eq:disint} are atomless for almost every base point.
\end{lemma}
\begin{proof}
Suppose first that $\beta_x^R$ has atoms on a set of positive measure. Select $F,a,\tau$ as above. The graph submeasure
\[
 d\Gamma_A(x,y)=1_F(x)a(x)\,dm(x)\,d\delta_{\tau(x)}(y)
\]
is dominated by $\Gamma_R$ and defines a positive bounded operator $0\leq A\leq R$, namely
\[
 Af(x)=1_F(x)a(x)f(\tau(x)).
\]
For clarity, any positive submeasure of $\Gamma_R$ defines such an operator: its input marginal is at most $\|R\|m$, and its output marginal is absolutely continuous with respect to $m$. Integration against an input $f\in L^1$ therefore gives an absolutely continuous output measure with variation at most $\|R\|\|f\|_1$. This also ensures that the displayed formula is well defined on $L^1$ classes.

Let $(r_n)$ be Rademacher functions on $(\T,m)$, with Borel representatives of modulus one. They are weakly null in $L^1(m)$, whereas
\[
 \|Ar_n\|_1=\int_F a\,dm>0.
\]
This contradicts the solid-sublattice property and $R\in\D$.

Suppose instead that $\alpha_y^R$ has an atomic part on a set of positive measure. A selection now gives a positive operator $0\leq A\leq R$ with
\begin{equation}\label{eq:reversegraph}
 \int gAf\,dm=\int_F a(y)f(y)g(\tau(y))\,dm(y).
\end{equation}
Put $\eta=\tau_*(a1_Fm)$. This is a nonzero measure absolutely continuous with respect to $m$, by the output marginal domination. In particular, $\eta$ is atomless. Choose Rademacher functions $(r_n)$ for the finite atomless measure $\eta$ and set
\[
 f_n(y)=1_F(y)r_n(\tau(y)).
\]
Then $f_n\rightharpoonup0$ in $L^1(m)$. Indeed, for $h\in L^\infty(m)$, the measure $\lambda_h=\tau_*(h1_Fm)$ satisfies
\[
 |\lambda_h|\leq\varepsilon^{-1}\|h\|_\infty\eta.
\]
Thus $\int hf_n\,dm=\int r_n\,d\lambda_h\to0$, since $r_n\rightharpoonup0$ in $L^1(\eta)$. If $q=d\eta/dm$, equation~\eqref{eq:reversegraph} gives $Af_n=qr_n$ and hence
\[
 \|Af_n\|_1=\eta(\T)>0.
\]
This again contradicts $A\in\D$.
\end{proof}

\begin{lemma}\label{lem:symbol}
Let $\Delta(x,y)=x-y$. For $R\geq0$,
\[
 \sigma(R)=\Delta_*\Gamma_R.
\]
For positive $S,T\in\B(X)$,
\begin{equation}\label{eq:product-symbol}
 \sigma(ST)=\int_\T\alpha_z^S*\check\beta_z^T\,dm(z).
\end{equation}
\end{lemma}
\begin{proof}
For bounded nonnegative $f,g$, averaging \eqref{eq:opmeasure} gives
\begin{align*}
 \int gE(R)f\,dm
 &=\int_\T\int_{\T^2}g(x-s)f(y-s)\,d\Gamma_R(x,y)\,dm(s)\\
 &=\int_{\T^2}\int_\T g(u)f(u-(x-y))\,dm(u)\,d\Gamma_R(x,y).
\end{align*}
This identifies the convolution symbol as $\Delta_*\Gamma_R$.

For the product, use the reverse kernel of $S$ and the forward kernel of $T$. The equality
\[
 \int gSTf\,dm
 =\int_\T\left(\int_\T g(x)\,d\alpha_z^S(x)\right)
             \left(\int_\T f(y)\,d\beta_z^T(y)\right)dm(z)
\]
holds first for bounded nonnegative $f,g$, by duality and disintegration. The corresponding measure is
\[
 d\Gamma_{ST}(x,y)=\int_\T d\alpha_z^S(x)\,d\beta_z^T(y)\,dm(z).
\]
Its total mass is finite because
\[
 \int_\T\alpha_z^S(\T)\beta_z^T(\T)\,dm(z)
 \leq\|S\|\int_\T T1\,dm\leq\|S\|\|T\|.
\]
Pushing this measure forward by $\Delta$ proves \eqref{eq:product-symbol}.
\end{proof}

\begin{proposition}\label{prop:annihilation}
For every $S,T\in\D$, $\restr{\sigma(ST)}{P}=0$.
\end{proposition}
\begin{proof}
If $S,T\geq0$, both factors in the integrand of \eqref{eq:product-symbol} are atomless almost everywhere, by Lemma~\ref{lem:atomless}. Lemma~\ref{lem:polar} and Tonelli's theorem give $\sigma(ST)(P)=0$. Positivity makes this equivalent to zero restriction. Decompose arbitrary real Dunford--Pettis operators into their positive and negative parts. For complex operators decompose their real and imaginary parts as well. The solid-sublattice property ensures all resulting positive operators remain Dunford--Pettis. Bilinearity then proves the assertion.
\end{proof}

\section{The square ideal and its quotient}\label{sec:square}

We first check explicitly how representable operators enter the construction. The Lewis--Stegall representation/factorization theorem identifies $\G$ with the operators factoring through $\ell^1$; see \cite{DU} and the discussion at the beginning of \cite[Section~3]{JS}. A representable $G$ has a strongly measurable essentially bounded kernel $k:\T\to L^1(\T)$ with
\[
 Gf=\int_\T f(y)k(y)\,dm(y).
\]

\begin{lemma}\label{lem:representable}
One has $\sigma(\G)\subseteq L^1(\T)m$, and
\[
 C_\mu\in\G\quad\Longleftrightarrow\quad\mu\ll m.
\]
Also, $\G\subseteq\I_2$.
\end{lemma}
\begin{proof}
Choose a jointly measurable representative $k(y,x)$ of the Bochner kernel. It is integrable on $\T^2$. A substitution in \eqref{eq:average}, justified by Fubini, gives
\[
 E(G)=C_{hm},\qquad h(r)=\int_\T k(t,t+r)\,dm(t),
 \qquad \|h\|_1\leq\int_\T\|k(t)\|_1\,dm(t).
\]
The changes of variables preserve product Haar measure, so the formula is independent of representatives. This proves the first assertion. If $\mu=hm$, its convolution operator is represented by the norm-continuous family of translates of $h$. Conversely, a representable convolution operator is fixed by $E$ and therefore has absolutely continuous symbol.

For the inclusion in the square ideal, take contractions $U:\ell^1\to X$ and $V:X\to\ell^1$ with $VU=\Id_{\ell^1}$. For example, choose disjoint measurable sets $A_n$ of positive measure and put
\[
 Ue_n=\frac{1_{A_n}}{m(A_n)},\qquad Vf=\left(\int_{A_n}f\,dm\right)_n.
\]
If $G=AB$ with $B:X\to\ell^1$ and $A:\ell^1\to X$, then
\[
 G=(AV)(UB).
\]
Both factors belong to $\G\subseteq\D$. This proves $\G\subseteq\I_2$, even without taking a closure at this step.
\end{proof}

\begin{lemma}\label{lem:powers}
For every $n\geq1$, the compact sumset $nP=P+\cdots+P$ has Haar measure zero. If $\mu$ is a Rajchman probability supported on $P$, every convolution power $\mu^{*n}$ is singular and Rajchman.
\end{lemma}
\begin{proof}
Strong independence with $k(P)=\infty$ implies that the subgroup $H$ generated by $P$ is torsion-free. Indeed, multiplying a finite integer representation of a torsion element by its order gives an integer relation among distinct points of $P$, so all coefficients vanish. Distinct elements $q,r$ of the torsion subgroup $\mathbb Q/\mathbb Z$ therefore give disjoint cosets $q+H$ and $r+H$. Consequently the countably many translates $q+nP$ are pairwise disjoint and have the same Haar measure. This forces $m(nP)=0$. The measure $\mu^{*n}$ is a probability supported on $nP$, and is Rajchman since $\widehat{\mu^{*n}}(k)=\widehat\mu(k)^n$.
\end{proof}

\begin{proof}[Proof of Theorem~\ref{thm:square}]
Define
\[
 \Theta:\D\longrightarrow M_0(P),\qquad
 \Theta(T)=\restr{\sigma(T)}{P}.
\]
This is a contraction by Proposition~\ref{prop:average} and Lemma~\ref{lem:rajchman}. Proposition~\ref{prop:annihilation} shows that it vanishes on $\I_2$. It satisfies $\Theta(C_\mu)=\mu$ for $\mu\in M_0(P)$. Hence, for every $W\in\I_2$,
\[
 \|C_\mu-W\|\geq\|\Theta(C_\mu-W)\|=\|\mu\|.
\]
Taking $W=0$ and using Lemma~\ref{lem:convolution} proves \eqref{eq:exact-distance}, and in particular $\I_2\ne\D$.

The set $\I_2$ is a closed two-sided ideal because $\D$ is such an ideal. Lemma~\ref{lem:representable} gives $\G\subseteq\I_2$. For a Rajchman probability $\mu$ supported on $P$, Lemma~\ref{lem:powers} shows that $\mu^{*2}$ is singular. Thus
\[
 C_\mu^2=C_{\mu^{*2}}\in\I_2\setminus\G,
\]
which proves the other strict inclusion. The same argument shows $C_\mu^n\notin\G$ for every $n\geq1$.

The map $\Theta$ induces a contraction $\widetilde\Theta:\D/\I_2\to M_0(P)$. Let
\[
 J\mu=C_\mu+\I_2.
\]
Then $J$ is contractive and $\widetilde\Theta J=\Id$, so $J$ is isometric and $J\widetilde\Theta$ is a contractive projection onto its range.

Finally fix the Rajchman probability $\mu$ on $P$. The map $f\mapsto f\mu$ is an isometry from $L^1(\mu)$ into $M_0(P)$, by Lemma~\ref{lem:rajchman}. Lebesgue decomposition relative to $\mu$ gives a contractive projection of $M_0(P)$ onto this subspace: the absolutely continuous part of any $\nu\in M_0(P)$ is again Rajchman by the same lemma. Since $\mu$ is atomless on a compact metric space, $L^1(\mu)$ is isometrically isomorphic to $L^1[0,1]$. Composing the projections gives the last assertion.
\end{proof}

\begin{corollary}\label{cor:quotient}
For $D=\D/\G$, one has $0\ne D^{[2]}\subsetneq D$. Neither $D$ nor $\D$ has a one-sided approximate identity. Moreover, $D/D^{[2]}$ has zero multiplication and contains the complemented subspaces described in Theorem~\ref{thm:square}.
\end{corollary}
\begin{proof}
Since $\G\subseteq\I_2$, quotienting gives $\I_2/\G=D^{[2]}$ and the isometric identification $\D/\I_2=D/D^{[2]}$. The strict inclusions follow from Theorem~\ref{thm:square}. Any one-sided approximate identity in an algebra makes the closed linear span of products equal to the algebra. This is impossible here. The multiplication in the quotient by the closed square is zero by definition.
\end{proof}

\section{Finite measure spaces with an atomless part}\label{sec:transfer}

The obstruction in Theorem~\ref{thm:ai} transfers through a complemented copy of $L^1[0,1]$. No separability of the ambient measure algebra is needed.

\begin{theorem}\label{thm:finite}
Let $(\Omega,\Sigma,\lambda)$ be a finite measure space whose measure algebra has a nonzero atomless part, and let $Z=L^1(\lambda)$. There is a norm-one $S\in\DP(Z)$ such that
\[
 \|S-ST\|\geq1,\qquad \|S-TS\|\geq1
 \qquad(T\in\DP(Z)).
\]
In particular, $\DP(Z)$ has neither a right nor a left approximate identity.
\end{theorem}
\begin{proof}
Choose an atomless measurable component of positive measure. Repeated equal-measure bisection produces a countably generated atomless subalgebra of its measure algebra. Conditional expectation onto this subalgebra, together with extension by zero from the component, yields a contractively complemented isometric copy $Y$ of $L^1[0,1]$ in $Z$. More explicitly, there are contractions $J:Y\to Z$ and $Q:Z\to Y$ with $QJ=\Id_Y$ and $J$ isometric; normalization of the restricted finite measure is absorbed in this isometry.

Take the operator $R$ from Theorem~\ref{thm:ai} on $Y$ and put $S=JRQ$. The ideal property gives $S\in\DP(Z)$. Since $QSJ=R$, one has $\|S\|=1$. For $T\in\DP(Z)$, put $F=QTJ\in\DP(Y)$. Then
\[
 Q(S-ST)J=R-RF,\qquad Q(S-TS)J=R-FR.
\]
Contractivity and Theorem~\ref{thm:ai} give the stated inequalities.
\end{proof}

\begin{remark}
If the finite measure algebra is purely atomic, its $L^1$ space is an $\ell^1$ space and has the Schur property. Its identity operator is therefore Dunford--Pettis. Thus the presence of an atomless part is precisely the obstruction to a one-sided approximate identity for the Dunford--Pettis ideal in this finite-measure setting. Theorem~\ref{thm:finite} does not assert that the square-ideal separation transfers through the same argument: compression of a product introduces intermediate operators on the whole ambient space.
\end{remark}

\section{Quantitatively norming power ideals}\label{sec:norming}

The properness of a closed square does not prevent it from controlling multiplication on an ambient quotient. We isolate the relevant abstract fact. For a closed two-sided ideal $J$ of a Banach algebra $A$, put
\[
 J^{[n]}=\clspan\{x_1\cdots x_n:x_i\in J\},\qquad n\geq1,
\]
and let $B_J$ denote its closed unit ball in the inherited norm.

\begin{theorem}\label{thm:norming}
Suppose $\kappa_r,\kappa_\ell>0$ satisfy, for all $a\in A$,
\begin{equation}\label{eq:detector}
 \kappa_r\|a\|\leq\sup_{x\in B_J}\|ax\|,\qquad
 \kappa_\ell\|a\|\leq\sup_{x\in B_J}\|xa\|.
\end{equation}
Then for every $n\geq1$,
\begin{equation}\label{eq:power-norming}
 \kappa_r^n\|a\|\leq\sup_{x\in B_{J^{[n]}}}\|ax\|\leq\|a\|,
 \qquad
 \kappa_\ell^n\|a\|\leq\sup_{x\in B_{J^{[n]}}}\|xa\|\leq\|a\|.
\end{equation}
In particular, every $J^{[n]}$ is essential on both sides, and left multiplication represents $A$ as a closed subalgebra of $\B(J^{[n]})$. Right multiplication gives the analogous representation of the opposite algebra $A^{\mathrm{op}}$.

If a bounded homomorphism $\Phi:A\to B$ satisfies
\[
 \|\Phi(x)\|\geq\eta\|x\|\qquad(x\in J^{[n]})
\]
for some $\eta>0$, and $A\ne\{0\}$, then
\begin{equation}\label{eq:transfer-lower}
 \|\Phi(a)\|\geq\frac{\eta\kappa_r^n}{\|\Phi\|}\|a\|
 \qquad(a\in A).
\end{equation}
\end{theorem}
\begin{proof}
Induct on $n$. Apply the right estimate in \eqref{eq:detector} to each $ax$, with $x\in B_{J^{[n]}}$, and then take suprema. Since $xy\in B_{J^{[n+1]}}$ whenever $y\in B_J$, this gives the next right estimate. The left estimate follows by the same induction, and the upper bounds follow from submultiplicativity.

The norm of left multiplication by $a$ on $J^{[n]}$ lies between $\kappa_r^n\|a\|$ and $\|a\|$. Its representation is therefore injective and bounded below, hence has closed range. Right multiplication reverses products, which accounts for the opposite algebra. If $I$ is a nonzero right ideal and $0\ne a\in I$, the right estimate produces $x\in J^{[n]}$ with $0\ne ax\in I\cap J^{[n]}$. The left-ideal assertion is symmetric. In fact, $x$ may be chosen with $\|x\|\leq1$ and $\|ax\|\geq\kappa_r^n\|a\|-\varepsilon$ for any $\varepsilon>0$.

Finally,
\[
 \eta\kappa_r^n\|a\|
 \leq\sup_{x\in B_{J^{[n]}}}\|\Phi(ax)\|
 \leq\|\Phi(a)\|\|\Phi\|,
\]
which proves \eqref{eq:transfer-lower}.
\end{proof}

\begin{lemma}\label{lem:ball}
Let $G\subseteq P$ be closed two-sided ideals in $A_0$, and let $q:A_0\to A_0/G$ be the quotient map. For $t\in A_0$,
\[
 \sup_{p\in B_P}\|q(tp)\|
 =\sup_{d\in B_{P/G}}\|q(t)d\|.
\]
The corresponding equality with left multiplication also holds.
\end{lemma}
\begin{proof}
One inequality follows from $q(B_P)\subseteq B_{P/G}$. Given $d\in B_{P/G}$ and $\varepsilon>0$, choose a representative $p\in P$ with $\|p\|\leq1+\varepsilon$. Then $p/(1+\varepsilon)\in B_P$, giving the reverse inequality after $\varepsilon\downarrow0$.
\end{proof}

\begin{corollary}\label{cor:conditional}
Suppose the operator estimates
\[
 \kappa_r\dist(T,\G)\leq
 \sup_{S\in B_\D}\dist(TS,\G),\qquad
 \kappa_\ell\dist(T,\G)\leq
 \sup_{S\in B_\D}\dist(ST,\G)
\]
hold for every $T\in\B(X)$. Then every $D^{[n]}$, for $D=\D/\G\subseteq A=\B(X)/\G$, satisfies \eqref{eq:power-norming}. In particular, the proper ideal $D^{[2]}$ of Theorem~\ref{thm:square} quantitatively norms $A$ on both sides.
\end{corollary}
\begin{proof}
Apply Lemma~\ref{lem:ball} and Theorem~\ref{thm:norming}.
\end{proof}

\begin{remark}
Corollary~\ref{cor:conditional} is expressly conditional on the displayed estimates; their proof is not part of this paper. Its right square-ideal constant is $\kappa_r^2$, hence $1/9$ if $\kappa_r=1/3$. Qualitative detection already gives algebraic faithfulness and essentiality of all powers. Quantitative estimates additionally give bounded-below representations and the transfer of norm lower bounds in \eqref{eq:transfer-lower}.
\end{remark}

\section{Further questions}

The construction gives $D^{[n]}\ne\{0\}$ for every $n$, because the same Rajchman probability has singular convolution powers of every order. It does not prove that
\[
 D\supsetneq D^{[2]}\supseteq D^{[3]}\supseteq\cdots
\]
is strictly decreasing beyond its first inclusion. Nor does the existence of the intermediate ideal classify the closed ideals between $\G$ and $\D$.

Another question is whether the exact square-ideal separation extends to every nonseparable finite atomless $L^1$ space. The complemented-copy argument proves Theorem~\ref{thm:finite} for all such spaces but does not establish that stronger assertion.

Finally, a Banach algebra is called incompressible if every bounded injective homomorphism from it into a Banach algebra is bounded below. Under the right estimate of Corollary~\ref{cor:conditional}, incompressibility of one $D^{[n]}$ would imply incompressibility of $A$, by restricting a homomorphism and applying \eqref{eq:transfer-lower}. Neither incompressibility nor its failure is proved here.

\section*{Acknowledgment of AI assistance}
The author used ChatGPT (OpenAI) to assist with drafting and revising the manuscript, literature searches, the exploration and critical examination of mathematical arguments, and LaTeX preparation. The author takes responsibility for the content, including the mathematical claims and references.

\end{document}